\documentclass[12pt,a4paper,twoside]{amsart}
\usepackage{amsmath}
\usepackage{amssymb}
\usepackage{amsfonts}

\numberwithin{equation}{section}

\date{}

\def\Ker{\text{\rm Ker}}

\def\Ga{\Gamma}
\def\bbr{\mathbb{F}}
\def\bbf{\mathbb{R}}

\def\bbn{\mathbb{N}}
\def\bbb{\mathbb{B}}

\def\half{\frac12}
\def\la{\lambda}

\def\suml{\sum\limits}
\theoremstyle{plain}
\newtheorem{thm}{Theorem}[section]
\newtheorem*{thm*}{Theorem}

\newtheorem{prop}[thm]{Proposition}

\theoremstyle{definition} 

\newtheorem*{definition*}{Definition}
\newtheorem*{thm1*}{Theorem A1}
\newtheorem*{thm2*}{Theorem A2}
\newtheorem{cor}[thm]
{Corollary}

\newtheorem*{claim*}{Claim}
\newtheorem{remark}[thm]{Remark}

\def\bbf{\mathbb{F}}
\def\bbr{\mathbb{R}}

\def\bbc{\mathbb{C}}
\def\bbh{\mathbb{H}}
\def\be{\begin{equation}}
\def\ee{\end{equation}}

\def\a{\alpha}
\def\vare{\varepsilon}
\def\ol{\overline}

\def\tle{\triangleleft}
\theoremstyle{remark}  
\begin{document}

\title{Finite simple groups of Lie type as expanders}
\author{Alexander Lubotzky}
\maketitle

\centerline{\it Dedicated to the memory of Beth Samuels who is deeply missed}

\baselineskip 16pt

\section{Introduction}

A finite $k$-regular graph $X,\; k\in\bbn$, is called an {\bf
$\vare$-expander} $(0< \vare \in\bbr)$, if for every subset of
vertices $A$ of $X$, with $|A| \le \half |X|$, $|\partial A|\ge
\vare |A|$ where $\partial A = \{ y\in X|$ distance $(y, A) =
1\}$.

The main goal of this paper is to prove:
\begin{thm} There exist $k\in\bbn$ and $0 < \vare \in\bbr$, such
that if $G$ is a finite simple group of Lie type, but not a Suzuki
group, then $G$ has a  set of $k$ generators $S$ for which the
Cayley graph $Cay(G;S)$ is an $\vare$-expander.

For short, we will say that these groups are uniform expanders or
expanders uniformly.
\end{thm}
Theorem 1.1 is new only for groups of small Lie rank: In
\cite{K1}, Kassabov proved that the family of groups
\begin{equation*} \{ SL_n(q)| 3 \le n \in\bbn,\;\; \,  q \text{ a\
prime\ power\ }\}\end{equation*} are uniform expanders. Nikolov
\cite{N}  proved that every {\bf classical} group is a bounded
product of $SL_n(q)$'s (with possible $n=2$, but the proof shows
that if the Lie rank is sufficiently high, say $\ge 14$, one can
use $SL_n(q)$ with $n\ge 3$). Bounded product of expander groups
are uniform expanders. Thus together, their results cover all
classical groups of high rank. So, our Theorem is new for
classical groups of small ranks as well as for the families of
exceptional groups of Lie type.

Theorem 1.1 gives the last step of the result conjectured in \cite{BKL} and announced in
\cite{KLN}:
\begin{thm}[\cite{KLN}] All  non-abelian finite simple groups, with the
possible exception of the Suzuki groups, are uniform expanders.
\end{thm}

By the classification of the finite simple groups, Theorem 1.1
covers all the simple groups except of finitely many sporadic
groups (for which the theorem is trivial) and the alternating
groups. The fact that  Theorem 1.2 holds for the alternating and
the symmetric groups is a remarkable result of Kassabov \cite{K2}.

The main new family covered by our method is $\{ PSL_2(q)|q\
\text{prime\ power\ }\}$. Unlike the result mentioned previously
whose proof used ingenious, but relatively elementary methods, the
proof for $PSL_2(q)$ will use some deep results from the theory of
automorphic forms. In particular, it will appeal to Selberg
$\lambda_1\ge \frac{3}{16}$ Theorem (\cite{Se},  see also
\cite[Chap.~4]{Lu}) and Drinfeld solution to the characteristic
$p$ Ramanujan conjecture ( \cite{Dr}).

For its importance, let us single it out as:
\begin{thm} The family $\{PSL_2(q)|q {\text \  prime\ power}\}$
forms a family of uniform expanders. \end{thm}

Let us mention right away that Theorem 1.3 was known before for
several subfamilies; e.g.\ for $\{ PSL_2(p)|p $ prime$\}$ (see
\cite[Chap.\ 4]{Lu} ] or $\{PSL_2(p^r) | p$ a fixed prime and
$r\in\bbn\}$ ( \cite{Mo}). The main novelty is to make them expanders uniformly for all $p $
and all $r$. To this end we will use the representation theoretic
reformulation of the expanding property (see \S 2) as well as the
new {\bf explicit} constructions of Ramanujan graphs in
\cite{LSV2} as special cases of Ramanujan complexes.  We stress
that the explicit construction there is crucial for our method and
not only the theoretical construction of \cite{LSV1}.
This will be shown in \S 3.

The case of $SL_2$ is a key step for the other groups of Lie type:
A result of Hadad (\cite{H1}, which is heavily influenced by
Kassabov \cite{K1}) enables one to deduce $SL_n(n\geq 2)$ from
$SL_2$.  Then in \S 4, we use a model theoretic argument to show
that simple groups of Lie type of bounded rank (including the
exceptional families except of the Suzuki groups) are bounded
products of $SL_2$'s.  Together with Nikolov's result mentioned
above, Theorem 1.1 is then fully deduced.

The Suzuki groups have to be excluded as they do not contain a copy
of $(P)SL_2(q)$ for any $q$, but we believe Theorem 1.2 holds for
them as well.

\medskip

\noindent{\bf Acknowledgment:} \  The author is indebted to
E.~Hrushovski,
Y.~Shalom and U.~Vishne for their useful advice,
and to the NSF, ERC and BSF (US-Israel) for their support.

\section{Representation theoretic reformulation}

It is well known (cf.\ \cite[Chap.\ 4]{Lu}) that expanding
properties of Cayley graphs $Cay(G; S)$ can be reformulated in the
language of the representation theory of $G$. For our purpose we
will need to consider also cases for which $S$ is not of bounded
size, in spite of the fact that our final result deals with
bounded $S$.  We therefore need a small extension of some standard
results, for which we need some notation:

The normalized adjacency matrix of a connected $k$-regular graph
$X$ is defined to be $\triangle = \frac 1k A$ where $A$ is the
adjacency matrix of $X$. The eigenvalues of $\triangle$ are in the
interval $[-1, 1]$. The largest eigenvalue in absolute value in
$(-1, 1)$ is denoted $\la(X)$.

For a group $G$, a set of generators $S$ and $\a > 0$, we denote
by $I(\a, G,S)$ the statement:

For every unitary representation $(V,\rho)$ of $G$, every $v \in
V$ and every $0 < \delta \in\bbr$, if $\| \rho(s) v-v\| < \delta$
for each $s\in S$, then $\|\rho(g) v-v\| < \a \delta$ for every
$g\in G$, (i.e., a vector $v$ which is ``$S$-almost invariant" is
also ``$G$-almost invariant".)

Note that the statement $I(\a, G, S)$ refers to all the unitary
 representations of $G$, wherever they have invariant vectors or
 not.

\begin{prop}


\noindent{\normalfont (i)} For every $\a > 0$ there is $\vare = \vare
(\a) > 0$ such that if $G$ is a finite group, $S$ a set of generators
and $I(\a, G, S)$ holds, then $Cay (G, S)$ is $\vare$-expander.
\vskip .06in

\noindent{\normalfont (ii)}  For every $\eta > 0$, there exists $\a = \a(\eta)$ such that if $G$ is a finite group with a set of
generators $S$, and $\la(Cay(G,S))<1-\eta$, then $I(\a, G,S)$
holds.
\vskip .06in

\noindent{\normalfont (iii)} If $k=|S|$ is bounded then the implications in {\normalfont (i)(a)} can be
reversed. (So $Cay(G,S)$ is expander iff every ``$S$-almost
invariant" vector is also ``$G$-almost invariant".)
\end{prop}

\begin{proof}

We note first that property $I(\a, G, S)$ implies that there
exists $\beta = \beta(\a) > 0$, such that for every unitary
representation $(V, \rho)$ of $G$ without a non-zero invariant
vector, and every $v \in V$ with $\| v \| = 1$, $\| \rho(s) v-v\|
\ge \beta$ for some $s \in S$.  Indeed, take $\beta <
\frac{1}{2\a}$ and so if $\| \rho (s) v-v \| < \beta$ for every $s
\in S$, then $I(\a, G, S)$ implies that $\| \rho (g) v-v\| < \frac
12$ for every $g\in G$.  This implies that $\ol v = \frac {1}{|G|}
\suml_{g \in G} \rho(g) v$ which is clearly a $G$-invariant
vector, is non-zero since $\| \ol v - v \| < \frac 12$. This
contradicts our assumption that $V$ does not contain an invariant
vector.

Altogether, $I(\a, G, S)$ implies the usual ``property $T$"
formulation and so the standard proof of Proposition 3.3.1 of
\cite{Lu} applies to deduce that $Cay(G,S)$ is an $\vare$-expander
for some $\vare = \vare(\beta(\a))$.  This proves (i). The
proof of (ii) is also a small modification of the standard
equivalences (see \cite[Theorem 4.3.2]{Lu}):  As it is well known,
a normalized eigenvalue gap (i.e.\ $\la(Cay (G, S)) < 1-\eta)$
implies an ``average expanding", i.e.\ if $(V,\rho)$ does not
contain an invariant vector, then
\begin{equation*}\tag{$\ast$} \frac{1}{|S|} \suml_{s\in S} \|
\rho(s) v-v\| \ge \eta' \| v \|
\end{equation*} (where $\eta'$ depends only on $\eta$). Note, that when $S$ is
unbounded, this is a stronger property than ``expanding" which
gives that for one $s\in S$, $\| \rho(s) v-v\| \ge \eta^{''} \| v
\|$ (for $\eta^{''} = \eta^{''} (\eta))$. Now, assume $(V, \rho)$
is an arbitrary unitary representation space of $G$ and $v \in V$,
of norm one, is $\delta$-invariant under $S$ for some $\delta <
\eta'$. Then $(\ast)$ implies that a large portion of $v$ is in
the space $V^G$ of $G$-fixed points. Hence $v$ is $G$-almost
invariant as needed.

Part (iii) is just the standard equivalences as in [Lu, Theorem
4.3.2].\end{proof}

An easy corollary of Proposition 2.1 is that `bounded products of expanders are expanders' or in a precise form:

\begin{cor} Let $G$ be a finite group and $G_i, i=1,\ldots, \ell$ a family of subgroups of $G$, each comes with a set of generators $S_i\subseteq G_i, i=1,\ldots, \ell$, with $\vert S_i\vert \leq r$.  Assume $G=G_1\cdot\ldots\cdot G_\ell$, i.e., every $g\in G$ can be written as $g=g_1g_2\ldots g_\ell$, with $g_i\in G_i$.  If all $Cay (G_i;S_i)$ are $\delta$-expanders, then $Cay (G;S)$ is $\vare$-expander for $S=\bigcup\limits^\ell_{i=1} S_i$ and $\vare$ which depends only on $\delta$ and $\ell$.
\end{cor}

\begin{proof}  If $(V,\rho)$ is a unitary representation of $G$, and $v\in V$ is a vector which is almost invariant under $S$, then it is almost invariant under each of the subgroups $G_i$ (by (2.1)(iii)) and as $G$ is a product of them, it is also almost invariant by $G$.  Now use (2.1)(i) to deduce the Corollary.
\end{proof}

Let us mention here another fact that will be used freely later.
The following Proposition is a special case of a more general
result in \cite{H2}:
 \begin{prop} Let $\{ G_i\}_{i \in I}$ be a family of perfect
finite groups (i.e. $[G_i, G_i]=G_i$) with sets of generators
$S_i$. Assume $\pi_i:\tilde G_i\to G_i$ is a central perfect cover
of $G_i$ and $\tilde S_i \subset \tilde G_i$ a subset for which
$\pi(\tilde S_i) = S_i$. If $Cay(G_i, S_i)$ are uniformly
expanders, then so are $Cay(\tilde G_i, \tilde S_i)$.
\end{prop}

The Proposition shows that proving uniform expanding for finite
simple groups or for their central extensions is the same problem.
So  a family of groups of the form $PSL_d(q)$ are expanders iff
$SL_d(q)$ are.

\section{$SL_2$: Proof of Theorem 1.3}

The goal of this section is to show that all the groups $\{
SL_2(q)|q \ \text{prime\ power }\}$ (and hence also $PSL_2(q)$)
are uniformly expanders. Let us recall
\begin{thm}
The Cayley graphs \newline $Cay (PSL_2(p)$;
$ \left\{ A= \left(\begin{matrix}1 & 1
\\0 &1\end{matrix}\right),\;  B = \left(\begin{matrix} 0 & 1\\-1
&0\end{matrix}\right)\right\} )$, for $p$ prime, are 3-regular
uniform expanders.
\end{thm}

For a proof, see [Lu, Theorem 4.4.2]. The proof uses Selberg
Theorem $\la_1\left(\Gamma(m)/\bbh^2\right) \ge \frac{3}{16}$ -
giving a bound on the eigenvalues of the Laplace-Beltrami operator
of the congruence modular surfaces. For a new method see [BG].

Another preliminary result needed is:
\begin{thm} {\normalfont (a)} For a fixed prime $p$, the groups $SL_2(p^k), \;
k \in \bbn$ have a symmetric subset $S_p$ of $p+1$ generators for
which the Cayley graphs $X=Cay (SL_2(p^k), S_p)$ are
$(p+1)$-regular Ramanujan graphs, i.e.  $\la(X) \le
\frac{2\sqrt{p}}{p+1}$.

{\normalfont (b)} The set of generators $S_p$ in part (a) can be
chosen to be of the form $\{ h^{-1} C h\, |\, h\in H\}$, where $C$
is some element of $SL_2(p^k)$ and $H$ is a fixed non-split torus
of $GL_2(p)$. (The proof will give a more detailed description of
$S_p$).
\end{thm}

Before proving Theorem 3.2, let us mention that part (a) has
already been proven by Morgenstern [Mo], but the specific form of
the generators as in (b) is crucial for our needs.  We therefore
apply to [LSV2] instead of [Mo].  We recall the construction
there: Let $\bbf_q$ be the field of order $q$ (a prime power),
$\bbf_{q^d}$ the extension of dimension $d$ and $\phi$ a generator
of the Galois group $Gal(\bbf_{q^d}/\bbf_q)$. Fix a basis $\{
\xi_0,\ldots, \xi_{d-1}\}$ for $\bbf_{q^d}$ over $\bbf_q$ where
$\xi_i = \phi^i (\xi_0)$. Extend $\phi$ to an automorphism of the
function field $k_1 = \bbf_{q^d}(y)$ by setting $\phi(y)=y$; the
fixed subfield is $k=\bbf_q(y)$, of codimension $d$.

Following the notation in [LSV2], we will denote by $R_T$ the ring
$\bbf_q [y, \frac{1}{1+y}]$ and for every commutative
$R_T$-algebra (with unit) $S$, we denote by $y$ the element $y\cdot
1 \in S$.  For such $S$ one defines an $S$-algebra $A(S) =
\mathop{\oplus}\limits^{d-1}_{i, j = 0} S\xi_i z^j$ with the
relations $z\xi_i = \phi(\xi_i) z$ and $z^d= 1+y$. Let $R = F_q
[y, \frac 1y, \frac{1}{1+y}] \subseteq k$ and denote $b = 1+z^{-1}
\in A(R)$. For every $u \in \bbf_{q^d}^*\subset A(R)^*$, we denote
$b_u = ubu^{-1}$.  As $\bbf_q^*$ is in the center of $A(R), b_u$
depends on the coset of $u$ in $\bbf_{q^d}^*/\bbf_q^*$.  This
gives $\frac{q^d-1}{q-1}$ elements $\{ b_u| u\in
\bbf_{q^d}^*/\bbf_q^*\}$ of $A(R)^*$. The subgroup of $A(R)^*$
generated by the $b_u$'s is denoted $\tilde \Gamma$ and its image
in $A(R)^*/R^*$ by $\Gamma = \Gamma_{d, q}$. For every ideal $I
\tle R$, we get a map
$$\pi_I: A(R)^*/R^* \to A(R/I)^* / (R/I)^*.$$
The intersection $\Ga\cap \Ker \pi_I$ is denoted $\Ga(I)$.

Theorem 6.2 of [LSV2] says:
\begin{thm} For every $d\ge 2$ and every $0 \neq I\tle R$, the
Cayley complex of $\Ga/\Ga(I)$ is a Ramanujan complex.
\end{thm}
The reader is referred to [LSV1] and [LSV2] for the precise
definition of Ramanujan complex and for the precise complex
structure of $\Ga/\Ga(I)$.  What is relevant for us here
is that this gives a spectral gap on the Cayley
graph of $\Ga/\Ga(I)$ with respect to the $\frac{q^d -1}{q-1} $
generators $ S = \{ b_u| u \in \bbf_{q^d}^*/\bbf_q^*\}$.

When $d=2$,\, $S$ is a symmetric set of generators of $\Gamma$ and
so $Cay(\Ga/\Ga(I); S)$ is a $k=(q+1)$-regular graph. When $d \ge
3$, $S \cap S^{-1} = \emptyset$ and $Cay (\Ga/\Ga(I);S)$ is a $k =
\frac{2(q^d - 1)}{q-1}$-regular graph. Let $A$ be its adjacency
matrix and $\Delta = \frac 1k A$ the normalized one.  Theorem 3.3
implies:
\begin{cor} Denote by $\mu_d$-the roots of unity in $\bbc$ of
degree $d$ and $E_d=\{ \frac{y +\bar y}{2}| y \in \mu_d\}$.  Let
$\lambda$ be an eigenvalue of $\Delta$.  Then either $\lambda \in
E_d$ or $|\lambda| \le \frac{dq^{(d-1)/2}}{(q^d-1)/(q-1)}$.
\end{cor}
\begin{remark}
Note that when $d=2,\,  k= |S| = q+1, E_d = \{ \pm 1\}$ and
Corollary 3.4 states that $Cay (\Ga/\Ga(I); S)$ are Ramanujan
graphs. The proof of this bound for $d=2$ requires Drinfeld
theorem (the Ramanujan conjecture for $GL_2$ over positive
characteristic fields). The proof for $d\ge 3$ is based on
Lafforgue's work [La]. It also requires the Jacquet-Langlands correspondence
in positive characteristic - see \cite{LSV1}, Remark 1.6). We mention, however, that for $d \ge 3$,
 quantitative estimates on Kazhdan property $(T)$ for
 $PGL_d(\bbf_q ((y)))$ give a weaker estimate such as: either
 $\lambda \in E_d$ or
 \[ \la \le \frac{1}{\sqrt{q}} + o(1)  \le
 \frac{19}{20},
 \]
which is valid for every $d$ and $q$. But only the case of $d=2$ needs the
deep results from the theory of automorphic forms.
\end{remark}

\begin{remark} The description above of the results from [LSV2]
brings only what is relevant to this paper. The bigger picture is as
follows: The group $A(R)^*/R^*$ is a discrete cocompact lattice in
$A(\bbf_q((y)))^*/\bbf_q((y))^*$.  The latter is isomorphic to
$H=P GL_d(\bbf_q((y)))$ and it acts on its Bruhat-Tits building
$\bbb$. The element $b \in H$ takes the initial point of the
building (the vertex $x_0$ corresponding to the lattice $\bbf_q
[[y]]^d$) to a vertex $x_1$ of distance one from it, where the
color of the edge $(x_0, x_1)$ is also one (so $x_1$ corresponds
to an $\bbf_q[[y]]$-submodule of $\bbf_q[[y]]^d$ of index $q$).
The group $\bbf_{q^d}^*/\bbf_q^*$ acts transitively on these
$(q^d-1)/(q-1)$ neighbors of $x_0$ of this type and the group
$\Ga$ generated by the $b_u$'s acts simply transitive on the
vertices of
 $\bbb$ - a result which goes back to Cartwright and Steger
 [CS]. For $b_u\in S$, $b^{-1}_u$ takes $x_0$ to a neighboring vertex
  of $x_0$ where the edge is of color $d-1$.  When $d=2$,\,  $d-1=1$,
  and $S$ is a symmetric set of size $q+1$ and Corollary 3.4 says
  that $Cay(\Ga/\Ga(I), S)$ are Ramanujan graphs.  For $d\ge 3$,
  $S\cap S^{-1} = \emptyset$ and $Cay (\Ga/\Ga(I), S)$ are regular
  graphs of degree $2|S| = \frac{2(q^d -1)}{q-1}$.
    The Ramanujan complex $\Ga/\Ga(I)$ is in fact isomorphic
 to the quotient $\Ga(I)\backslash \bbb$ of the Bruhat-Tits
 building. On the building $\bbb$ (and on its quotients
 $\Ga(I)\backslash \bbb$) we have an action of $d-1$ Hecke
 operators $A_1, \cdots, A_{d-1}$ and the Ramanujan property
 gives bounds on their eigenvalues. For $d \ge 3$, $A_1 + A_{d-1}$
 is nothing more than the adjacency operator of the Cayley graph
 of $\Ga/\Ga(I)$, and for $d=2, A_1 = A_{d-1}$ and $A_1$ is the
 adjacency operator.
\end{remark}

 The structure of the quotient group $\Ga/\Ga(I)$ is analyzed in
 [LSV2]; If $I$ is a prime ideal of $R$ with $R/I \simeq
 \bbf_{q^e}$, then $\Ga/\Ga(I)$ is isomorphic to a subgroup of
 $PGL_d(q^e)$ containing $PSL_d(q^e)$.  Theorem 7.1 of [LSV2] gives
 a more precise description, showing that essentially all
 subgroups between $PSL_d(q^e)$  and $PGL_d(q^e)$ can be obtained if
 $I$ is chosen properly. The image of $S$ in $PGL_d(q^e)$ which we
 also denote by $S$ is composed of one element $C$, the image of $b$
 in the notations above, and the conjugates of $C$ by the
 non-split tori in $PGL_d(q^e)$ of order $(q^d-1)/(q-1)$.

 Note that $PGL_d(q^e)/PSL_d(q^e)$ is a cyclic group and the image
 of $S$ there is a single element - the image of $C$, since all
 the other elements of $S$ are conjugates of $C$.  The eigenvalues
 in $E_d$ above, may appear as ``lift up" of the eigenvalues of
 the cyclic group generated by $C$ (which is a subgroup of
 $PGL_d(q^e)/PSL_d(q^e)$ and a quotient of $\Ga/\Ga (I)$) whose
 order divides $d$. These eigenvalues will be called ``the trivial
 eigenvalues" of $Cay(\Ga/\Ga(I); S)$ and they are in the subset
 $E_d$ defined in Corollary 3.4.  If the image of $\Ga$ in
 $PGL_d(q^e)$ is only  $PSL_d(q^e)$, then $1$ is the only trivial
 eigenvalue of $\Delta$ and all the others satisfy the bound of
 Corollary 3.4.

 For $d$ large the issue of which subgroup of $PGL_d(q^e)$ is
 obtained is somewhat delicate. For $d=2$, which is what is needed here for
 the proof of Theorem 3, Theorem 7.1 of [LSV2] ensures
 that for $p^e > 17, PSL_2 (p^e)$ can be obtained, if $I$ is
 chosen properly. Thus \newline $Cay (PSL_2 (p^e), S)$ are $(p+1)$-regular
 Ramanujan graphs. They are, therefore, also $\vare$-expanders by
 Proposition 2.1(ii), but with an unbounded number of generators.

 We now show that this is true also with a bounded number of
 generators. An explicit form of Theorem 1.3 is:

 \begin{thm}  The family of Cayley graphs $Cay(PSL_2(\ell); \{ A,
 B, C, C'\})$, when $\ell = p^e$ is any prime power, are uniformly
 expanders. (Here $A = \begin{pmatrix} 1 &1\\0&1\end{pmatrix}, B
 = \begin{pmatrix} 0 &1\\ -1 &0\end{pmatrix}$,  $C$ is as in the
 description above when $d=2$ and $q=p$ and $C'$ will be described in the proof).
 \end{thm}

 \begin{proof} Let $C$ be as described above (with $d=2$ and $q=p$
  a prime).  The image of $S$ in $PSL_2(p^e)$ as described above,
  is the set of conjugates of $C$ under the action of the
  non-split torus $T$ of $PGL_2(p)$ which is isomorphic to $
  \bbf^*_{p^2}/ \bbf^*_p$ and of order $p+1$.  Denote $T_1 = T\cap
  PSL_2(p)$ a subgroup of index at most 2 in $T$ (in fact index 2,
  unless $p=2$). Let $C$ and $C'$ be two representatives of the
  orbits of $S$, under conjugation by $T_1\colon C$ as before and $C'$
  a representative of the other orbit (if exists).

  We can now prove the Theorem by using Proposition 2.1: Let $(V,\rho)$ be a
 unitary representation of $PSL_2(\ell)$, with an $\{ A, B,
 C, C'\}$-almost invariant vector $v$. Restrict the representation
 $\rho$ to the subgroup $PSL_2(p)$. By Theorem 3.1, $Cay
 (PSL_2(p); \{ A, B\})$ are expanders and hence by Proposition 2.1(iii),
 $v$ is $PSL_2(p)$ almost invariant. As it is also $C$-almost invariant,
 it is almost invariant under the set $PSL_2(p)\cdot C\cdot
 PSL_2(p)$ and similarly with $PSL_2(p) \cdot
 C'\cdot PSL_2(p)$. The union of these last two sets contain $S$. So, $v$ is $S$-almost
 invariant. But $\newline \la(Cay (PSL_2(p^e), S)) \le
 \frac{2\sqrt{p}}{p+1} < \frac{19}{20}$ for every $p$ and $e$, so
 by Proposition 2.1(ii), $v$ is $PSL_2(p^e)$-almost invariant
 and by Proposition 2.1(i), $Cay (PSL_2(p^e); \{ A, B, C, C'\})$
 are uniform expanders. This finishes the proof of Theorem 3.7
 (and hence also of Theorem 1.3).
 \end{proof}

 Recall that by Proposition 2.3, Theorem 3.7 also says that the
 family $\{ SL_2(\ell)|\ell$ $\text{\ a prime\  power} \}$ is
 family of expanders.  Let us now quote:
 \begin{thm}[Hadad {[H1]}, Theorem 1.2]. Let $R$ be a finitely generated ring with
 stable range $r$ and assume that the group $EL_d(R)$ for some $d
 \ge r$ has Kazhdan constant $(k_0, \vare_o)$.  Then there exist
 $\vare = \vare (\vare_0) > 0$ and $k = k(k_0) \in\bbn$ such that
 for every $n\ge d, EL_n(R)$ has Kazhdan constant $(k, \vare)$.
\end{thm}

 We refer the reader to [H1] for the proof. We only mention here
 that if $R$ is a field then its stable range is 1 and $EL_n(R)$,
 the group of $n\times n$ matrices over $R$ generated by the
 elementary matrices, is $SL_n(R)$. Also
 recall that a finite group $G$ has Kazhdan constant $(k, \vare)$
 if it has a set of generators $S$ of size at most $k$, such that for every
 non-trivial irreducible representation $(V,\rho)$ of $G$ and for
 every $0\neq v \in V$, there exists $s\in S$ such that $\|
 \rho(s) v-v\| \ge \vare \| v \|$. As is well known, this
 implies that $Cay (G;S)$ is $\vare'$-expander for some $\vare'$
 which depends only on $\vare$. All these remarks combined with
 Theorems 3.7 and 3.8 give:
\begin{thm} The groups $\Big\{ SL_n(q)\; \Big| \; \text{all \   } 2 \le
n \in\bbn, \;  q \ \text{ prime\ power\ } \Big\} $ form a family
of expanders uniformly.
\end{thm}

\begin{remark} In the proof of Theorem 3.9, we used for $d\geq 3$, Theorem 3.8 of Hadad whose proof was heavily influenced by Kassabov's proof \cite{K1} that all $SL_n(q), n\geq 3$, are expanders. So our proof cannot be considered as a really different proof for $n\geq 3$. In \cite{KLN}, a second very different proof for $SL_n, n\geq 3$ was announced, based on the theory of Ramanujan complexes. But it turns out that the proof sketched there has a mistake, for which the author of the current paper takes full responsibility. The idea there was to handle $SL_d$, $d$ even, say $d=2m$, by using the following argument: Corollary 3.4 above gives a spectral gap with respect to an unbounded subset $S$ of $(P)GL_d(q)$ which consists of conjugates of a single element by a non-split tori $T$. This $T$ as a subgroup of $G=GL_d(q)$ is inside a copy of $H = GL_2(q^m)$. Passing from $G$ to $PGL_d(q)$, $T$ is then in the image $\bar H$ of $H$. We argued there that by dividing by the center, $T\subset \bar H$ and $\bar H$ is isomorphic to $PGL_2(q^m)$. (We then wanted to use Theorem 1.3 for $\bar H$ to deduce that $\bar H$ is an expander and to continue to argue as in the proof of Theorem 3.7). It is not true however, that $\bar H$ is $PGL_2(q^m)$: we divided by the center of $G$ which is of order at most $d$ and not by the center of $H$ which is of order $q^m-1 > > d$. So $\bar H$ has a large abelian quotient and it is far from being an expander.
\end{remark}

\section{Bounded generation by $SL(2)$}

A finite group $G$ is said to be a product of $s$ copies of
$SL_2$, if there exist prime powers $q_i$ and homomorphisms
$\varphi_i: SL_2(q_i) \to G$, $ i = 1, \ldots, s$, such that for
every $g\in G$ there exist $x_i \in SL_2(q_i)$, $ i = 1, \ldots,
s$ with $g = \varphi_1(x_1) \cdot\ldots\cdot \varphi_s(x_s)$.

Theorem 3.7 shows that all the groups $SL_2(q)$ are uniform
expanders (with 4 generators for each one). It now follows
from Corollary 2.2 that for a fixed $s$, all
the groups which are products of $s$ copies of $SL_2$ are uniform
expanders with $4s$ generators. We will now show that this is
indeed the case for all finite simple groups of Lie type of
bounded rank, excluding the groups of Suzuki type.
\begin{thm} There exists a function $f:\bbn\to\bbn$, such that if
$G$ is  a finite simple group of Lie type of rank $r$, but not of
Suzuki type, then it is a product of $f(r)$ copies of $SL_2$.
\end{thm}

Before giving the proof we remark that Theorem 4.1 combined with
Theorem 3.9 and the result of Nikolov \cite{N} implies Theorem
1.1. Indeed, by \cite{N}, a classical group of Lie type is a
bounded product of groups of type $SL_n(q)$ ($n$ and $q$ varies)
and so by Theorem 3.9 they are uniform expanders.  The other
finite simple groups of Lie type  have bounded rank and so are
bounded product of $SL_2$ by Theorem 4.1, and hence also uniform
expanders. This excludes, of course, the Suzuki groups for
which every homomorphism from $SL_2(q)$ to a Suzuki group has
trivial image since the order of the latter is not divisible by 3.
Thus the validity of Theorem 1.2 for the Suzuki groups is left
open.

Back to Theorem 4.1. This result has been announced in \cite{KLN}
and a model theoretic proof based on the work of Hrushovski and
Pillay \cite{HP} was sketched there.  Recently,  Liebeck, Nikolov and Shalev
\cite{LNS}   proved the theorem by standard group
theoretic arguments. This is somewhat more technical and requires
some case by case analysis but has the advantage that they came
out with an explicit function $f(r)$ which is valid for every
group $G$ of rank $r$. This is of importance for our application
for expanders as it enables one to deduce explicit $k$ and $\vare$
in Theorem 1.1.

Anyway, we will bring here the model theoretic proof. For a nice introduction to the model theory of
finite simple groups, see \cite{W}.  As there are only finitely
many group types of bounded rank, we can take $G$ to be a fixed
(twisted or untwisted) Chevalley group and we need to prove the
result for the groups $G(F)$ when $F$ is a finite field.  We will
show below that each such $G(F)$ contains a copy of $(P) SL_2(F)$ as a
{\it uniformly  definable subgroup} . By a definable subgroup, we mean a
subgroup that can be defined using a first order sentence in the
language of rings with a distinguished endomorphism  - the language
in which $G$ is defined. By uniformly definable we mean that the
subgroup $(P)SL_2(F)$ is defined by a single sentence  -
independent of $F$.

Assuming this fact, we can argue as follows: Let $F_i$ be an
infinite family of finite fields and $K= (\Pi F_i)/U$ an ultra
product of them, i.e., $U$ is a non-principle ultra-filter.  Thus
$K$ is a pseudo-algebraically closed field (PAC, for short -- see
\cite{FJ} and \cite{HP}).
Let
$\tilde G = \Pi G(F_i)/U$ the corresponding ultra product of the
groups $G(F_i)$. By a basic result (Point \cite{P}, Propositions 1
and 2 and Corollary 1) $\tilde G$ is a simple group isomorphic to
$G(K)$ and similarly the ultra-product of the $(P)SL_2(F_i)$'s
gives a subgroup $(P)SL_2(K)$ of $\tilde G = G(K)$.

Now, as $\tilde G = G(K)$ is simple, it is generated by the
conjugates of $(P)SL_2(K)$. By \cite[Proposition 2.1]{HP} $G(K)$
is a product of $m < \infty$ conjugates of $(P)SL_2(K)$.
This is an elementary statement about $G(K)$ and hence it is
true also for $G(F_i)$ for almost all $i$.  This proves what we
need modulo the promised fact.
\begin{remark} The model theoretic proof gives (when one follows
the arguments in \cite{HP}) that $m \le 4\dim G$.  Moreover, in
principal one can give an explicit bound $M$ such that the above
claim is true for every $F$ with $ |F| > M$.  The proof in
\cite{LNS} gives explicit bounds on $m$ which are usually (but not
always) slightly better and are valid for all $F$.
\end{remark}

We are left with proving our claim that $G(F)$ contains a copy of
$(P)SL_2(F)$ as a uniformly definable subgroup.

If $G$ splits (i.e. untwisted type), e.g. $G = E_6$, it contains
$SL_2$ as a  subgroup generated by a root subgroup and its
opposite. Note that a root subgroup is  definable and as $SL_2$ is
a bounded product of the root subgroup and its opposite, it is
also definable.  Of course, in this case it is even an algebraic
subgroup.

If $G$ is twisted, but not a group of Ree type (i.e. all the
simple roots of $G$ are of the same length, so the type is $A_n,
D_n$ or $E_6$), e.g., look at $G(q) = ^{2}\!\!E_6 (q)$. Then $G$ is
the group of points of $E_6(q^2)$ of the following form: $\{ g\in
E_6(q^2)\mid g^{F_{r}} = g^\tau\}$ where $\tau$ is the graph
automorphism of $E_6$ and $F_r$ the Frobenius automorphism. By
restriction of scalars, this is an algebraic group defined over
$\bbf_q$. If the automorphism $\tau$ has a fixed vertex, e.g. for
our example ${}^2E_6$, then ${}^2E_6(q)$ contains a copy of
$SL_2(q)(\subseteq SL_2(q^2)\subseteq E_6(q^2))$ corresponding to
this vertex and as an $\bbf_q$-group- this is an algebraic
subgroup. The argument we illustrated here with ${}^2E_6(q)$ works
equally well with  the other twisted groups with  fixed vertex (of
course, for $^{3}D_4$ we should take $D_4(q^3)$ - but the rest is
the same). This covers all the cases except of $A_n, n$ even. But
${}^2A_n$ is anyway $SU(n+1)$ which contains $SU(2)$ and it is
well known that $SU(2, q^2)$ is isomorphic to $SL_2 (q)$.

We are left with the twisted groups of Ree type:
$^{2}F_4(2^{2n+1})$ and $^{2}G_2(3^{2n+1})$ (the other type
$^{2}B_2(2^{2n+1})$ give the Suzuki groups and these were excluded
from the theorem). Now, $^{2}F_4(2^{2n+1})$ is known (cf. [GLS]
Table 2.4 VI, Table 2.4.7, Theorems 2.4.5 and 2.48) to have a
subgroup generated by a root subgroup and its opposite which is
isomorphic to $SL_2 (2^{2n+1}$).  (This is not the case for all
roots; for some we get the Suzuki groups, but we need only one
root which gives $SL_2$).
 For $^{2}G_2(3^{2n+1})$ one can argue by pure group theoretical
terms: it is known (cf. \cite{E}) to have a unique conjugacy class
of involutions and if $\tau$ is such an involution, then
$C_G(\tau)$ - the centralizer of $\tau$ - is isomorphic to $H =
\langle\tau\rangle\times PSL_2(3^{2n+1})$.  Within $H, PSL_2
(3^{2n+1})$ is the set of all commutators of $H$ (since every
element of $PSL_2(q)$ is a commutator). Thus $PSL_2(3^{2n+1})$ is
a definable subgroup in a uniform way of $^{2}G_2(3^{2n+1})$. The proof of Theorem
4.1 (and hence of 1.1) is now complete.





\vskip .50in

Institute of Mathematics

Hebrew University

Jerusalem, 91904 ISRAEL

alexlub@math.huji.ac.il

\end{document}